\documentclass[11pt,reqno]{amsart}

\usepackage{amsbsy,amsfonts,amsmath,amssymb,amscd,amsthm,mathrsfs,verbatim,calc}
\usepackage{graphicx,epsfig,color,tikz}

\usepackage{enumitem}
\usepackage[a4paper,text={6.1in,8in},centering,includefoot,foot=1in]{geometry}
\usepackage[colorlinks=true,hypertexnames=false,            linkcolor=blue,citecolor=green]{hyperref}

\small\normalsize
\theoremstyle{plain}

\newtheorem{theorem}{Theorem}[section]
\newtheorem{lemma}[theorem]{Lemma}

\newtheorem{proposition}[theorem]{Proposition}
\newtheorem{corollary}[theorem]{Corollary}

\theoremstyle{definition}

\newtheorem{remark}[theorem]{Remark}

\numberwithin{equation}{section}
\let\oldmarginpar\marginpar
\renewcommand\marginpar[1]{\-\oldmarginpar[\raggedleft\footnotesize \textcolor{red}{#1}]{\raggedright\footnotesize\textcolor{red}{#1}}}

\DeclareMathOperator{\diam}{diam}

\begin{document}
\title[Bounds for (strong) Roman $k$-dominations]{Bounds for (strong) Roman $k$-dominations}

\author[F. Khosh-Ahang]{Fahimeh Khosh-Ahang Ghasr}
\address{{Department of Mathematics, Ilam University,
P.O.Box 69315-516, Ilam, Iran.}}
\email{f.khoshahang@ilam.ac.ir and fahime$_{-}$khosh@yahoo.com}
\begin{abstract}
Motivated by resource defense models in networks, such as protecting territories with varying legion strengths, let $k \geq 2$ be an integer. Roman $k$-domination and strong Roman $k$-domination generalize Roman, double Roman, Italian, and double Italian domination to arbitrary number of legions. The main goal of this note is establishing sharp upper bounds for the Roman and strong Roman $k$-domination numbers of connected graphs. These bounds unify and extend prior results for $k=2$ and $k=3$. We also precisely characterize the graphs achieving these bounds.

%sina googooly yadet bashe 1403\9\18
%man sina hastam gheily naz hastam 1404\3\7
%maman 10goololo ast
\end{abstract}


%	  05E40   Combinatorial aspects of commutative algebra
%     05E45   Combinatorial aspects of simplicial complexes
%     05C75   Structural characterization of families of graphs
%     05C25   graphs and abstract algebra (groups, rings, ...)
%     05C50   graphs and linear algebra
%     05C40   connectivity
%     05C69   vertex subsets with special properties (dominating sets, independent sets, cliques , etc.)

%	  06A11   Algebraic aspects of posets

%     11C20   Matrices , determinants in number theory
%     11C08   polynomials and matrices

%     13P10   Grobner basis other bases for ideals and modules
%     13C14   Cohen-macaulay modules
%	  13H10   Special types (Cohen-Macaulay, Gorenstein, Buchsbaum, etc.)
%	  13D02   Syzygies, resolutions, complexes
%	  13A02   Graded rings
%  	  13F20   Polynomial rings and ideals; rings of integer-valued polynomials
%	  13A18   Valuations and their generalizations
%     13C40   linkage, complete intersection and determinantal ideal

%	  14M25   Toric varieties, Newton polyhedra [See also 52B20]

%	  16S36   Ordinary and skew polynomial rings and semigroup rings


\maketitle
\setcounter{tocdepth}{1}
%\tableofcontents

\section{Introduction}

Throughout, let $k$ be an integer with $k\geq 2$ and $G = (V, E)$ a finite connected simple graph of order $n$. For a vertex $v$ in $V$, denote its open neighbourhood by $N_G(v)$ and closed neighbourhood by $N_G[v] = \{v\} \cup N_G(v)$, with \textbf{degree} $\deg v = |N_G(v)|$. For a subset $S$ of $V$, $G\setminus S$ is a graph obtained from $G$ by removing all vertices in $S$ and their incident edges. If $S=\{v\}$, we denote $G\setminus S$ by $G\setminus v$.

A subset $S$ of $V$ is a \textbf{dominating set} if $V = \bigcup_{v \in S} N_G[v]$, with the \textbf{domination number} $\gamma(G)$ of $G$ being the minimum cardinality of such sets in $G$. Domination theory traces back to the mid-19th century and has evolved extensively~\cite{HHS}. Roman domination and its variants (e.g., Italian, double Roman, double Italian, perfect, weak, etc.) emerged in the late 20th century and continue to develop~\cite{double,Chellali,perfect,Henning+hedetniemi,Klostermeyer2,Li,MV,Stewart,V}; see surveys~\cite{Chellali2,Chellali3,Klostermeyer2,Klostermeyer1}. These concepts model scenarios like defending territories with legions, as Stewart's analogy to the Roman Empire~\cite{Stewart}, where vertices represent locations and legions represent resources with allocation constraints.

In~\cite{Fahimeh}, (strong) Roman $k$-domination generalizes these concepts, allowing up to $k$ legions per vertex to model flexible resource allocation in critical networks. It is worth noting that the literature already contains work on domination parameters involving an integer $k$. In particular  \cite{triple,khalili,MV} introduce and study a notion called $[k]$-Roman domination. Although the notions resembles that of \cite{Fahimeh}, the underlying definitions differ substantially: $[k]$-Roman domination restricts the labels to a $k$-element
set and imposes a different domination condition, whereas the current paper
follows the generalization introduced in \cite{Fahimeh}, allowing weights up to $k$ and
imposing neighbourhood-sum conditions based on $k$. These represent two
distinct and non-equivalent extensions of classical Roman domination. To
avoid confusion for readers we just recall definition of (strong) Roman $k$-domination from \cite{Fahimeh}:

For the finite simple graph $G$, a function $f: V \to \{0, 1, \dots, k\}$ assigns a \textbf{weight} $f(v)$ to each vertex $v$ in $V$. Let $V_i = \{v \in V \mid f(v) = i\}$ for $i \in \{0, 1, \dots, k\}$. Then $f$ can be denoted as $f = (V_0, V_1, \dots, V_k)$. The \textbf{weight} of $f$ is $w(f) = \sum_{v \in V} f(v)$. The function $f$ is a \textbf{Roman $k$-dominating function} ($k$-RDF) of $G$ if, for every $u \in V$ with $f(u) < k/2$, we have $\sum_{v \in N_G[u]} f(v) \geq k$. It is a \textbf{strong Roman $k$-dominating function} ($k$-SRDF) of $G$ if, for every $u \in V$ with $f(u) <k/2$, we have $f(u) + \sum_{v \in N_G(u), f(v) > k/2} f(v) \geq k$. The \textbf{Roman $k$-domination number} is
$$\gamma_k(G) = \min \{ w(f) \mid f \text{ is a } k\text{-RDF of } G \},$$
and the \textbf{strong Roman $k$-domination number} is
$$\gamma_k^s(G) = \min \{ w(f) \mid f \text{ is a } k\text{-SRDF of } G \}.$$
A $k$-RDF (resp.\ $k$-SRDF) $f$ of $G$ with $w(f) = \gamma_k(G)$ (resp.\ $w(f) = \gamma_k^s(G)$) is called a $\gamma_k(G)$-function (resp.\ $\gamma_k^s(G)$-function). Notably:
\begin{itemize}
\item Roman 2-domination equals Italian domination~\cite{Chellali},
\item Strong Roman 2-domination equals Roman domination~\cite{coc},
\item Roman 3-domination equals double Italian domination~\cite{MV},
\item Strong Roman 3-domination equals double Roman domination~\cite{double}.
\end{itemize}
Thus, studying Roman $k$-domination for arbitrary integers $k \geq 2$ unifies prior work and provides new insights into parity-dependent bounds (arising from legion indivisibility in odd $k$ cases).
Upper bounds for $\gamma_k(G)$ and $\gamma_k^s(G)$ when $k\in \{2,3\}$  appear in~\cite{double,chambers,HHV,Klos}, with extremal graphs characterized.
This paper derives unified sharp upper bounds for $\gamma_k(G)$ and $\gamma_k^s(G)$ for arbitrary $k\geq 2$ when $n\geq 3$ providing new bounds for $k\geq 4$, that extend these results, as summarized in the table below for comparison.
\begin{table}[htbp]
  \centering
  \begin{tabular}{c|c|c}
    \hline
    $k$ & Bound for $\gamma_k(G)$ & Bound for $\gamma_k^s(G)$ \\
    \hline
    2  & $3n/4$ (\cite{HHV,Klos}) & $4n/5$ (\cite{chambers})\\
    3 & $5n/4$ & $5n/4$ (\cite{double})\\
    4,6  & $3kn/8$ & $2kn/5$\\
    Even, $\geq 8$ & $3kn/8$ & $3kn/8$ \\
    Odd, $\geq 5$ & $(3k+1)n/8$ & $(3k+1)n/8$\\
    \hline
  \end{tabular}
  \vspace{2mm}
    \caption{Comparision of bounds for $\gamma_k(G)$ and $\gamma_k^s(G)$}\label{table}
\end{table}

For connected simple graphs $G$ with $n \geq 3$ vertices, Corollary~\ref{3.5} yields the bounds. Theorems~\ref{3.7} and~\ref{3.9} precisely characterize graphs attaining these bounds, extending prior results for $k=2$ and $k=3$.

\section{Basic Results on Roman \(k\)-Domination}

We begin by bounds on the strong Roman $k$-domination number in terms of the domination number, leveraging the following key property.

\begin{lemma}\cite[Lemma 2.4]{Fahimeh}\label{2.4F}
Suppose that $f=(V_0, \dots , V_k)$ is a $\gamma_k^s(G)$-function. Then we may assume $|V_i|=0$ for each integer $i$ with $0<i<k/2$.
\end{lemma}
\begin{proposition}\label{Prop2.2}
For a graph $G$, if we set $A_k=\lfloor\frac{k+1}{2}\rfloor$, then
$$A_k\gamma(G)\leq \gamma_k^s(G)\leq k\gamma(G),$$
\end{proposition}
\begin{proof}
If $S$ is a dominating set for $G$ with $|S|=\gamma(G)$, then assigning $k$ to each $x$ in $S$, and zero to other vertices yields a $k$-SRDF for $G$. This shows $\gamma_k^s(G)\leq k\gamma(G)$. Conversely, if $f=(V_0, \dots , V_k)$ is a $\gamma_k^s(G)$-function, then by Lemma \ref{2.4F}, we may assume that $|V_i|=0$ for $0 <i <k/2$. Furthermore $\bigcup_{A_k\leq i \leq k} V_i$  is a dominating set for $G$. Therefore
$$\gamma_k^s(G)=w(f)=\sum_{0\leq i \leq k} i|V_i|= \sum_{A_k\leq i \leq k} i|V_i| \geq \sum_{A_k\leq i \leq k} A_k|V_i|=A_k\sum_{A_k\leq i\leq k} |V_i| \geq A_k\gamma (G).$$
\end{proof}

Setting $k=2,3$ in Proposition \ref{Prop2.2} recovers \cite[Proposition 8]{double} and \cite[Proposition 1]{coc} concerning Roman and double Roman domination numbers. In the following result, which generalizes Theorem 2.5 and Proposition 2.6 of \cite{Fahimeh}, we obtain some upper bounds for $\gamma_k(G)$ and $\gamma_k^s(G)$ by exploiting known upper bounds for smaller values of $k$. In the subsequent section,  we present an alternative approach based on the explicit construction of suitable k-RDFs and k-SRDFs, which yields upper bounds that, in several cases, are strictly sharper than those established here in Theorem \ref{bound}(4).
\begin{theorem} \label{bound}
\begin{enumerate}
  \item For positive integers $c$ and $k$
$$\gamma_{ck}(G)\leq c\gamma_k(G), \qquad  \gamma_{ck}^s(G) \leq c\gamma_k^s(G).$$
  \item  For each $\gamma_k(G)$-function $f=(V_0,\dots , V_k)$ and  positive integers $c, k$ we have
$$\gamma_{ck+1}(G)\leq c\gamma_k(G)+\sum_{0\leq i \leq A_k}|V_i|\leq c\gamma_k(G)+|V|,$$
where $A_k=\lfloor k/2\rfloor$.
  \item For each $\gamma_k^s(G)$-function $f=(V_0,\dots , V_k)$ and  positive integers $c,k$ we have
$$\gamma_{ck+1}^s(G)\leq c\gamma_k^s(G)+|V|-|V_0|\leq (c+1)\gamma_k^s(G).$$
  \item For even integers $k$
$$\gamma_k(G)\leq 3kn/8, \ \ \gamma_k^s(G)\leq 2kn/5,$$
and for odd integers $k$
$$\gamma_k(G)\leq (3k+5)n/8, \ \ \gamma_k^s(G)\leq 2(k+1)n/5.$$
\end{enumerate}
\end{theorem}
\begin{proof}
(1) Let $c,k\in \mathbb{N}$ and $f$ be a $\gamma_k(G)$-function (resp. a $\gamma_k^s(G)$-function).
Define a map $f'\colon V \to \{0,\dots ,ck\}$ by $f'(u)=cf(u)$ for each vertex $u$ of $G$.
It is straightforward to verify that $f'$ is a $(ck)$-RDF (resp. a $(ck)$-SRDF) of $G$.
Hence the result holds.
  
(2) Let $c,k\in \mathbb{N}$, $f=(V_0,\dots , V_k)$ be a $\gamma_k(G)$-function and $f'=(V'_0, \dots , V'_{ck+1})$, where $V'_{c\ell+1}=V_\ell$ for each $\ell$  with $0\leq \ell \leq A_k$, $V'_{c\ell}=V_\ell$ for each $\ell$ with $A_k+1\leq \ell \leq k$ and $V'_i=\emptyset$ for all other integers $i$. When $c=1$, $V'_{A_k+1}=V_{A_k}\cup V_{A_k+1}$  and define the remaining $V'_i$ as above. We first show that $f'$ is a $(ck+1)$-RDF. Let $u\in V(G)$ with $f'(u)<(ck+1)/2$.
Then $u\in V'_{c\ell+1}=V_\ell$ for some $\ell$ with $0\leq \ell <k/2$, since
$c\ell+1<(ck+1)/2$ implies $\ell<k/2-1/(2c)$.
Thus $f(u)<k/2$, and hence $\sum_{v\in N_G[u]}f(v)\geq k$.
Therefore,
\begin{align*}
  \sum_{v\in N_G[u]}f'(v) &= f'(u)+\sum_{v\in N_G(u), 0\leq f(v)\leq A_k}f'(v)+\sum_{v\in N_G(u),A_k+1\leq f(v) \leq k }f'(v) \\
   &= cf(u)+1+\sum_{v\in N_G(u), 0\leq f(v)\leq A_k}(cf(v)+1)+\sum_{v\in N_G(u),A_k+1\leq f(v) \leq k }(cf(v))  \\
  & \geq c\sum_{v\in N_G[u]}f(v)+1\\
  &\geq ck+1.
\end{align*}
Consequently,
\begin{align*}
\gamma_{ck+1}(G)&\leq w(f')\\
&=\sum_{0\leq i \leq ck+1}i|V'_i|\\
&=\sum_{0\leq \ell \leq A_k}((c\ell+1)|V_\ell |)+\sum_{A_k+1\leq \ell \leq k}(c\ell |V_\ell |)\\
&=c\sum_{0\leq \ell \leq k}(\ell |V_\ell |)+\sum_{0\leq \ell \leq A_k}|V_\ell |\\
&=c\gamma_k(G)+\sum_{0\leq \ell \leq A_k}|V_\ell |\\
&\leq c\gamma_k(G)+|V|.
\end{align*}

(3) Let $c,k\in \mathbb{N}$, $f=(V_0,\dots , V_k)$ be a $\gamma_k^s(G)$-function and $A_k=\lfloor (k+1)/2 \rfloor$. Since $f$ is a $\gamma_k^s(G)$-function by \cite[Lemma 2.4]{Fahimeh} we may assume $V_i=\emptyset$ for each $i$ with $0<  i< k/2$.  Hence
$$\sum_{A_k\leq \ell \leq k}(\ell |V_\ell |)=\gamma_k^s(G), \qquad \sum_{A_k\leq \ell \leq k}|V_\ell |=|V|-|V_0|.$$

 Let $f'=(V'_0, \dots , V'_{ck+1})$, where $V'_0=V_0$, $V'_{c\ell+1}=V_\ell$ for each $\ell$  with $A_k\leq \ell \leq k$  and $V'_i=\emptyset$ for other integers $i$. We show that $f'$ is a $(ck+1)$-SRDF. Let $u\in V(G)$ with $f'(u)<(ck+1)/2$. Then $u\in V_0$, since otherwise $u\in V'_{c\ell+1}$ for some $\ell$  with $A_k\leq \ell \leq k$, which would imply $c\ell+1<(ck+1)/2$, and hence $\ell<A_k$, a contradiction. Thus $\sum_{v\in N_G(u), f(v)>k/2} f(v) \geq k$. Therefore, by \cite[Remark 2.2(8)]{Fahimeh},
 \begin{align*}
 f'(u)+\sum_{v\in N_G(u), f'(v)>(ck+1)/2}f'(v)&=\sum_{v\in N_G(u), f'(v)>(ck+1)/2}f'(v)\\
 &=\sum_{v\in N_G(u), f(v)>k/2}(cf(v)+1)\\
 &\geq c\sum_{v\in N_G(u), f(v)>k/2}f(v)+1\\
 &\geq ck+1.
 \end{align*}
 Hence
\begin{align*}
\gamma_{ck+1}^s(G)&\leq w(f')\\
&=\sum_{1\leq i \leq ck+1}i|V'_i|\\
&=\sum_{A_k\leq \ell \leq k}((c\ell+1)|V_\ell |)\\
&=c\sum_{A_k\leq \ell \leq k}(\ell |V_\ell |)+\sum_{A_k\leq \ell \leq k}|V_\ell |\\
&=c\gamma_k^s(G)+|V|-|V_0|.
\end{align*}
Moreover,
$$|V|-|V_0|=\sum_{1\leq i \leq k} |V_i| \leq \sum_{1\leq i\leq k} i|V_i|=\gamma_k^s(G),$$
which yields the second inequality.

(4) The result follows directly from the known upper bounds for $\gamma_2(G)$ and $\gamma_2^s(G)$.
\end{proof}

The following two lemmas are pivotal for Section~3. The first demonstrates that attaching more than two pendant edges to a vertex does not affect $\gamma_k(G)$ and $\gamma_k^s(G)$.

\begin{lemma}\label{3.1}
 Assume $G$ is a graph, $v\in V(G)$, $\mathcal{L}_{v,G}$ is the set of all leaves in $G$ adjacent to $v$ and $r=|\mathcal{L}_{v,G}|$.
\begin{enumerate}
\item If $r\geq 2$, then there exists a $\gamma_k(G)$-function (resp. $\gamma_k^s(G)$-function) $f$ with $f(v)=k$ and $f(u)=0$ for each leaf adjacent to $v$.
\item Let $r\geq 2$ and $H$ be a graph obtained from $G$ by attaching some pendant edges to $v$. Then $\gamma_k(H)= \gamma_k(G)$ (resp. $\gamma_k^s(H)= \gamma_k^s(G)$).
\end{enumerate}
\end{lemma}
\begin{proof}
(1) Suppose $g$ is a $\gamma_k(G)$-function (resp. $\gamma_k^s(G)$-function) of $G$. Either $g(u)<k/2$ for some $u\in \mathcal{L}_{v,G}$ or $g(u)\geq k/2$ for all $u\in \mathcal{L}_{v,G}$. Hence, either $g(u)+g(v)\geq k$ for some $u\in \mathcal{L}_{v,G}$, or $\sum_{u\in \mathcal{L}_{v,G}} g(u)\geq k$. Thus $g(v)+\sum_{u\in \mathcal{L}_{v,G}}g(u)\geq k$. Define $f(v)=k$, $f(u)=0$ for all $u\in \mathcal{L}_{v,G}$ and $f(x)=g(x)$ for other vertices $x$, then $f$ is a $k$-RDF (resp. $k$-SRDF) for $G$ with $w(f)\leq w(g)$. So $f$ is a $\gamma_k(G)$-function (resp. $\gamma_k^s(G)$-function) as required.

(2) Since $r\geq 2$, by Part 1, there exists a $\gamma_k(G)$-function (resp. $\gamma_k^s(G)$-function) $f$ with $f(v)=k$ and $f(u)=0$ for all $u\in \mathcal{L}_{v,G}$. Define $f'(w)=0$ for each new leaf $w$ adjacent to $v$ in $H$ and $f'(x)=f(x)$ for other vertices $x$. Then $f'$ is a $k$-RDF (resp. $k$-SRDF) for $H$. So $\gamma_k(H)\leq w(f')=w(f)=\gamma_k(G)$ (resp. $\gamma_k^s(H)\leq\gamma_k^s(G)$). Conversely, since $|\mathcal{L}_{v, H}| \geq |\mathcal{L}_{v, G}|\geq 2$, by Part 1, there exists a $\gamma_k(H)$-function (resp. $\gamma_k^s(H)$-function)  $f'$ with $f'(v)=k$ and $f'(u)=0$ for all $u\in \mathcal{L}_{v,H}$. If $f$ is  the restriction of $f'$ to $V(G)$, then $f$ is a $k$-RDF (resp. $k$-SRDF) for $G$. Thus
$\gamma_k(G)\leq w(f)=w(f')=\gamma_k(H)$ (resp. $\gamma_k^s(G)\leq\gamma_k^s(H)$).
\end{proof}

The next lemma explores subdivided pendant edges' impact on (strong) Roman $k$-domination under certain conditions.
\begin{lemma}\label{3.2}
Assume $H$ is a graph, $r$ is a non-negative integer and $u$ is a vertex in $H$ with $r$ pendant edges. Let $G$ be the graph obtained by subdividing every possible pendant edge $uw_i$ in $H$ by a vertex $v_i$. 
\begin{enumerate}
\item If $k$ is even and every $\gamma_k(G)$-function $f$ has $f(u)<k/2$, then $r=0$.
\item If $k$ is odd and every $\gamma_k^s(G)$-function $f$ has $f(u)<k/2$, then $r\leq 1$.
\end{enumerate}
\end{lemma}
\begin{proof}
(1) Suppose, to the contrary, that $r\geq 1$ and $f$ is a $\gamma_k(G)$-function. For each $1\leq i \leq r$ we have $f(v_i)+f(w_i)\geq k-f(u)$, since one of the following holds:
\begin{itemize}
\item $f(w_i)<k/2$. So $f(v_i)+f(w_i)\geq k\geq k-f(u)$.
\item $f(w_i)\geq k/2, f(v_i)\geq k/2$. So $f(v_i)+f(w_i)\geq k\geq k-f(u)$.
\item $f(w_i)\geq k/2, f(v_i)< k/2$. So $f(v_i)+f(w_i)\geq k-f(u)$.
\end{itemize}
Combining these with $f(u)<k/2$ and $r\geq 1$ implies
\begin{align*}
f(u)+\sum_{1\leq i\leq r}(f(v_i)+f(w_i))& \geq f(u)+r(k-f(u))\\
%&=rk-(r-1)f(u)\\
&\geq 
%rk-(r-1)k/2\\&=
(r+1)k/2.
\end{align*}
Define $f'(w_i)=f'(u)=k/2, f'(v_i)=0$ for each $i$ with $1\leq i \leq r$ and $f'(x)=f(x)$ for other vertices $x$. Then $f'$ is a $k$-RDF for $G$ with $w(f')\leq w(f)$. So $f'$ is a $\gamma_k(G)$-function with $f'(u)=k/2$, a contradiction.

(2) Suppose, to the contrary, that $r\geq 2$ and $f$ is a $\gamma_k^s(G)$-function. By Lemma \ref{2.4F}, we may assume  $f(u)=0$. Thus, for each $i$ with $1\leq i \leq r$, $f(v_i)+f(w_i)\geq k$. Since $r\geq 2$ and $k\geq 3$,
\begin{align*}
f(u)+\sum_{1\leq i\leq r}(f(v_i)+f(w_i))&\geq rk\\
&\geq (r+1)(k+1)/2.
\end{align*}
Define $f'(u)=f'(w_i)=(k+1)/2, f'(v_i)=0$ and $f'(x)=f(x)$ for other vertices $x$. Then $f'$ is a $k$-SRDF for $G$ with $w(f')\leq w(f)$. So $f'$ is a $\gamma_k^s(G)$-function with $f'(u)\geq k/2$, a contradiction.
\end{proof}

\section{Upper Bounds for (Strong) Roman $k$-Domination Number}

We commence by introducing special graph concepts central to the remainder of the paper. The \textbf{centers} (or Jordan centers) of a graph are vertices $u$ minimizing the greatest distance $d(u,v)$ to any vertex $v$. The \textbf{diameter}, $\diam G$, is the longest distance $d(u,v)$ for any vertices $u,v$ in $G$, and a \textbf{maximal path} is a path that cannot be extended by adding any more vertices or edges. Let $r$ and $s$ be positive integers. A tree with a single center adjacent to $r$ vertices is called a \textbf{star}, denoted $S_r$. A tree with two centers, one adjacent to $r$ leaves and the other to $s$ leaves, is a \textbf{double star}, denoted $S_{r,s}$. For a positive integer $t$ a \textbf{wounded spider} is a star graph $S_t$ with at least one and at most $t-1$ edges subdivided, while a \textbf{healthy spider} has all edges subdivided. In a wounded or healthy spider, the center of $S_t$  is the \textbf{head}, vertices at distance two from the head are \textbf{healthy feet} and vertices at distance one from the head are \textbf{wounded feet}. The path and cycle graphs with $r$ vertices are respectively denoted by $P_r$ and $C_r$.

We now generalize \cite[Theorem 13]{double}, \cite[Theorem 2.1]{chambers} and \cite[Theorem 12]{Klos} to arbitrary integer $k\geq 2$.

\begin{theorem}\label{3.3}
Let $k\geq 2, n\geq 3$ and $T$ be a tree with $n$ vertices. Then
$$\gamma_k(T)\leq \left\lbrace \begin{array}{ll}
\cfrac{3k+1}{8}n &  \text{ when } k \text{ is odd}, \\
\cfrac{3k}{8}n &  \text{ when } k \text{ is even};
\end{array}
\right. \text{and} \ \ \gamma_k^s(T)\leq \left\lbrace \begin{array}{ll}
\cfrac{3k+1}{8}n &  \text{ when } k \text{ is odd}, \\
\cfrac{2k}{5}n &  \text{ when } k\in \{2,4,6\},\\
\cfrac{3k}{8}n &  \text{ when } k\geq 8 \text{ and } k \text{ is even}.
\end{array}
\right.$$
\end{theorem}
\begin{proof}
First we deal with $\gamma_k^s(T)$. For $k\in \{2,4,6\}$, set $A_k=2k/5$; for other even $k$, set $A_k=3k/8$; and for odd $k$, set $A_k=(3k+1)/8$. We proceed by induction on $n$.

If $\diam T=2$, then $T=S_r$ for some $r$ with $r\geq 2$. Hence $\gamma_k^s(T)=k< A_kn$ since $n\geq 3$.

If $\diam T=3$, then $T$ is the double star $S_{r,s}$ for some positive integers $r,s$. If $r=1$ or $s=1$, then assign zero to one center, $k/2$ (resp. $(k+1)/2$) to its sole neighbour, $k$ to the other center, and zero to all its  neighbours to obtain a $k$-SRDF for $T$. Thus  $\gamma_k^s(T)=3k/2$ (resp. $(3k+1)/2$)$\leq A_kn$ since $n\geq 4$. If $r,s>1$, assign $k$ to both centers and zero elsewhere, yielding a $k$-SRDF with $\gamma_k^s(T)=2k < A_kn$ since $n\geq 6$.

Therefore we may assume $\diam T\geq 4$ and so $n\geq 5$. Let $P$ be a maximal path in $T$ with end vertices $w_0$ and $r$. Root $T$ at $r$, let $v_0$ be the vertex in $P$ adjacent to $w_0$ and $u$ its parent. Fix $P, r, u, v_0$ and $w_0$ in such a way that $\deg v_0$ is maximum. The following cases arise:

\textbf{Case I.}  $\deg v_0>3$.
By the choice of $v_0$, Lemma \ref{3.1}(2) and the inductive hypothesis,
$$\gamma_k^s(T)=\gamma_k^s(T\setminus w_0)\leq A_k(n-1)<A_kn.$$

\textbf{Case II.} $\deg v_0=3$. Suppose $w_0$ and $w_0'$ are the only children of $v_0$.  Since \linebreak $T'=T\setminus \{v_0, w_0, w'_0\}$ is a tree with at least three vertices, the inductive hypothesis gives $\gamma_k^s(T')\leq A_k(n-3)$. If $f'$ is a $\gamma_k^s(T')$-function, define $f(v_0)=k, f(w_0)=f(w'_0)=0$ and $f(x)=f'(x)$ for other vertices $x$. Then $f$ is a $k$-SRDF for $T$ and
$$\gamma_k^s(T)\leq w(f)=w(f')+k=\gamma_k^s(T')+k\leq A_k(n-3)+k<A_kn.$$

\textbf{Case III.} $\deg v_0=2$ and $T$ is a spider with head $u$. Since $\diam T\geq 4$, $T$ is a spider as in Figure \ref{fig2} with $a+2$ healthy feet and $b$ wounded feet for some non-negative integers $a$ and $b$. When $a=b=0$, we have $T=P_5$ and direct computation shows that $\gamma_k^s(T)=2k=A_kn$ for $k\in\{2,4,6\}$ while $\gamma_k^s(T)=3k/2+3\leq A_kn$ for other even $k$ (resp. $\gamma_k^s(T)=(3k+3)/2<A_kn$).  When $b=0$ and $k$ is odd, assign $(k+1)/2$ to the head and healthy feet and zero elsewhere; in other cases, assign $k$ to the head, $k/2$ for even $k$ (resp. $(k+1)/2$) to the healthy feet and zero elsewhere, yielding a $k$-SRDF with weight less than $A_kn$.

\textbf{Case IV.} $\deg v_0=2$ and $T$ is not a spider. Suppose $T'$ is the induced subgraph of $T$ on $u$ and its descendants and  $T''=T\setminus V(T')$. Then $T'$ is a spider with head $u$ and $T''$ is a tree with at least three vertices. So by Cases III and $\diam T\leq 3$ and inductive hypothesis, there are $k$-SRDFs $f'$ and $f''$  respectively for $T'$ and $T''$ such that $w(f')\leq A_k|V(T')|$ and $w(f'')\leq A_k|V(T'')|$. By defining $f(x)=f'(x)$ for $x\in V(T')$ and $f(x)=f''(x)$ for $x\in V(T'')$, one may find a $k$-SRDF $f$ on $T$. So
 $$\gamma_k^s(T)\leq w(f)=w(f')+w(f'')\leq A_k|V(T')|+A_k|V(T'')|=A_kn.$$
 
 To prove the bounds for $\gamma_k(T)$, notice $\gamma_k(T)\leq \gamma_k^s(T)$. So the bounds for $\gamma_k^s(T)$ also apply to $\gamma_k(T)$ and it remains to prove $\gamma_k(T)\leq 3kn/8$ for $k\in \{2,4,6\}$. The proof is similar to above except when $T=P_5$. Direct computation shows when $T=P_5$, $\gamma_k(T)=3k/2 < A_kn$. 
\end{proof}

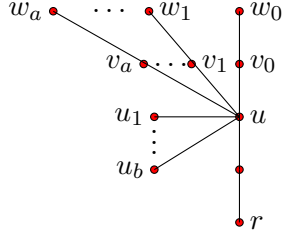
\begin{figure}[htbp]
\centering
\begin{tikzpicture}
[scale=0.7]
\filldraw[fill=red] (0,0) circle (2pt) node[right]{$r$};
\filldraw[fill=red] (0,1) circle (2pt);
\filldraw[fill=red] (0,2) circle (2pt) node[right]{$u$};
\filldraw[fill=red] (-1.6,2) circle (2pt) node[left]{$u_1$};
\filldraw[fill=red] (-1.6,1) circle (2pt) node[left]{$u_b$};
\filldraw[fill=red] (-1.6,1.1) node[above]{$\vdots$};

\filldraw[fill=red] (0,3) circle (2pt) node[right]{$v_0$};
\filldraw[fill=red] (-0.9,3) circle (2pt) node[right]{$v_1$};
\filldraw[fill=red] (-1.8,3) node[right]{$\dots$};
\filldraw[fill=red] (-1.8,3) circle (2pt) node[left]{$v_a$};
\filldraw[fill=red] (0,4) circle (2pt) node[right]{$w_0$};
\filldraw[fill=red] (-1.7,4) circle (2pt) node[right]{$w_1$};
\filldraw[fill=red] (-2.4,3.8)  node[above]{$\dots$};
\filldraw[fill=red] (-3.5,4) circle (2pt) node[left]{$w_a$};

\draw (0,2) -- (-1.7,4);
\draw (0,2) -- (-3.5, 4);
\draw (0,2) -- (-1.6, 2);
\draw (0,2) -- (-1.6,1);

\draw (0,1) -- (0,2);
\draw (0,2) -- (0,3);
\draw (0,3) -- (0,4);
\draw (0,0) -- (0,1);
\end{tikzpicture}

 \caption{A spider with head $u$  in the proof of Theorem \ref{3.3}}
\label{fig2}
\end{figure}

%Since $\gamma_k(T)\leq \gamma_k^s(T)$, the bounds for $\gamma_k^s(T)$ in Theorem \ref{3.3} also apply to $\gamma_k(T)$. However, the next theorem provides a sharper bound for $\gamma_k(T)$ when $k\in\{2,4,6\}$.

%\begin{theorem}\label{3.4}
%Let $n\geq 3$, $k$ be even and $T$ be a tree with $n$ vertices. Then $\gamma_k(T)\leq 3kn/8$.
%\end{theorem}
%\begin{proof}
%The proof is similar to that of Theorem \ref{3.3} except when $T=P_5$. Direct computation shows when $T=P_5$, $\gamma_k(T)=3k/2 < (3k/8)\cdot 5$. 
%\end{proof}

Since each connected graph has a spanning tree and adding edges can't increase the value of (strong) Roman
$k$-domination number, Theorem \ref{3.3} yields:
\begin{corollary}\label{3.5}
If $n\geq 3$ and $G$ is a connected graph with $n$ vertices, then

$$\gamma_k(G)\leq \left\lbrace \begin{array}{ll}
\cfrac{3k+1}{8}n &  \text{ when } k \text{ is odd}, \\
\cfrac{3k}{8}n &  \text{ when } k \text{ is even};
\end{array}
\right. \text{and} \ \
\gamma_k^s(G)\leq \left\lbrace \begin{array}{ll}
\cfrac{3k+1}{8}n &  \text{ when } k \text{ is odd}, \\
\cfrac{2k}{5}n &  \text{ when } k\in\{2,4,6\},\\
\cfrac{3k}{8}n &  \text{ when } k\geq 8 \text{ and } k \text{ is even}.
\end{array}
\right.
$$
\end{corollary}

Now, we consider the cases where equalities in Corollary \ref{3.5} hold. To this end, recall that the rooted product $H(\mathcal{K})$ of a graph $H$ by a family of rooted graphs $\mathcal{K}=\{K_x \ | \ x\in V(H)\}$ is defined in 1978, \cite{G+M}, as the union of $H$ and $K_x$s such that for each vertex $x$ of $H$ one should identify $x$ with the root of $K_x$. Now suppose $\mathcal{B}$ is a set of rooted graphs whose root is one of their centers. A graph $G$ is a $\mathcal{B}$-branch graph if $G=H(\mathcal{K})$ for some connected graph $H$ such that $K_x$s belong to $\mathcal{B}$. In this case any $K_x$ is called a branch of $G$. If $G$ is also a tree, it is called a $\mathcal{B}$-branch tree. A $\{P_4, P_5\}$-branch graph is shown in Figure \ref{fig4}.

\vspace{1mm}
\begin{figure}[htbp]
\centering
\begin{tikzpicture}
[scale=0.6]
\filldraw[fill=red] (-2.5,0) circle (2pt);
\filldraw[fill=red] (-2.5,-1) circle (2pt);
\filldraw[fill=red] (-2.5,-2) circle (2pt);
\filldraw[fill=red] (-3,-1) circle (2pt);

\draw (-3,-1) -- (-2.5,0);
\draw (-2.5,0) -- (-2.5,-1);
\draw (-2.5,-1) -- (-2.5,-2);

\filldraw[fill=red] (-1,0) circle (2pt);
\filldraw[fill=red] (-1,-1) circle (2pt);
\filldraw[fill=red] (-1,-2) circle (2pt);
\filldraw[fill=red] (-1.5,-1) circle (2pt);
\filldraw[fill=red] (-2,-2) circle (2pt);

\draw (-1,0) -- (-1,-1);
\draw (-1,0) -- (-1.5,-1);
\draw (-1,-1) -- (-1,-2);
\draw (-1.5,-1) -- (-2,-2);

\draw (-0.2,-0.8) node[below]{$\dots$};

\filldraw[fill=red] (1,0) circle (2pt);
\filldraw[fill=red] (1,-1) circle (2pt);
\filldraw[fill=red] (1,-2) circle (2pt);
\filldraw[fill=red] (0.5,-1) circle (2pt);

\draw (0.5,-1) -- (1,0);
\draw (1,0) -- (1,-1);
\draw (1,-1) -- (1,-2);

\filldraw[fill=red] (2.5,0) circle (2pt);
\filldraw[fill=red] (2.5,-1) circle (2pt);
\filldraw[fill=red] (2.5,-2) circle (2pt);
\filldraw[fill=red] (2,-1) circle (2pt);
\filldraw[fill=red] (1.5,-2) circle (2pt);

\draw (1.5,-2) -- (2,-1) -- (2.5,0) -- (2.5,-1) -- (2.5,-2);

\draw (0,-0.4) node[above]{$H$};
\draw[blue] (0, 0) ellipse (4cm and 0.5cm);

\end{tikzpicture}
\caption{A $\{P_4, P_5\}$-branch graph when $H$ is a connected graph  }
\label{fig4}
\end{figure}
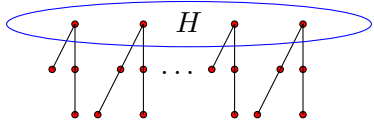

\begin{lemma}\label{3.6}
Let $n\geq 3$ and $G$ be a graph with $n$ vertices.
\begin{enumerate}
\item If $G=C_5$, then $\gamma_8^s(G)=15$ and $\gamma_k^s(G)=2k$ when $k\in \{2,4,6\}$.
\item If $G$ is a $\{P_5\}$-branch graph, then $\gamma_8^s(G)=3n$ and $\gamma_k^s(G)=2kn/5$ when $k\in \{2,4,6\}$.
\item If $G$ is a $\{P_4\}$-branch graph, then $\gamma_k^s(G)=3kn/8$ for all even integers $k$ with $k\geq 8$.
\item If $G$ is a $\{P_4\}$-branch graph, then $\gamma_k^s(G)=(3k+1)n/8$ for all odd integers $k$.
\item If $G$ is a $\{P_4,P_5\}$-branch graph, then $\gamma_8^s(G)=3n$.
\end{enumerate}
\end{lemma}
\begin{proof}
Consider $C_5$, $P_4$ and $P_5$ as:
$$V(C_5)=\{u_1, \dots , u_5\}, E(C_5)=\{u_1u_2, u_2u_3, u_3u_4, u_4u_5, u_5u_1\},$$
$$V(P_4)=\{\ell_1, c, r_1, r_2\}, E(P_4)=\{\ell_1c, cr_1, r_1r_2\},$$
$$V(P_5)=\{\ell_2,\ell_1, c, r_1, r_2\}, E(P_5)=\{\ell_2\ell_1, \ell_1c, cr_1, r_1r_2\}.$$

(1) For $k\in \{2,4,6\}$ assigning $f(u_1)=k, f(u_2)=f(u_5)=0, f(u_3)=f(u_4)=k/2$ and for $k=8$ assigning $f(u_1)=f(u_3)=f(u_4)=k/2+1, f(u_2)=f(u_5)=0$  define desired $k$-SRDF with minimum weight for $C_5$.  

(2) Suppose $k\in \{2,4,6\}$ (resp. $k=8$), $G$ is a $\{P_5\}$-branch graph and $f$ is a $\gamma_k^s(G)$-function. Consider an arbitrary branch of $G$.
 If $f(c)\leq k/2$, then we should have $f(\ell_1)+f(\ell_2)\geq k$ and $f(r_1)+f(r_2)\geq k$. Else the following cases occur for the left side of the branch:
\begin{itemize}
\item $f(\ell_2)<k/2$. Then $f(\ell_1)+f(\ell_2)\geq k$.
\item $f(\ell_2)=k/2, f(\ell_1)<k/2$. Then $f(\ell_1)+f(c)\geq k$.
\item $f(\ell_2)\geq k/2, f(\ell_1)\geq k/2$. Then $f(\ell_1)+f(\ell_2)\geq k$.
\item $f(\ell_2)>k/2, f(\ell_1)<k/2$. Then $f(\ell_1)+f(\ell_2)+f(c)\geq k$.
\end{itemize}
Similar cases occur for the right side of the branch. So we may consider the following for the branch:
\begin{itemize}
\item[i.] $f(\ell_1)+f(\ell_2)\geq k$ and $f(r_1)+f(r_2)\geq k$.
\item[ii.] $f(\ell_1)+f(\ell_2)\geq k$ and $f(r_1)+f(c)\geq k$.
\item[iii.] $f(\ell_1)+f(\ell_2)\geq k$ and $f(r_1)+f(r_2)+f(c)\geq k$.
\item[iv.] $f(\ell_1)+f(c)\geq k$ and either $f(r_1)+f(c)\geq k$ or $f(r_1)+f(r_2)+f(c)\geq k$. This is the case when $f(\ell_2)=k/2$, $f(r_2)\geq k/2, f(\ell_1)<k/2, f(r_1)<k/2$. In view of Lemma \ref{2.4F} we may assume that $f(\ell_1)=f(r_1)=0$. So $f(c)\geq k$.
\item[v.] $f(\ell_1)+f(\ell_2)+f(c)\geq k$ and $f(r_1)+f(r_2)+f(c)\geq k$. This is the case when $f(\ell_1)<k/2, f(r_1)<k/2, f(\ell_2)>k/2, f(r_2)>k/2, f(c)>k/2$. Hence $f(\ell_2)+f(r_2)+f(c)\geq 3k/2+3\geq 2k$ (resp. $15$) since $k\leq 6$ (resp. $k=8$).
\end{itemize}
Therefore in all of the above cases we have $f(\ell_2)+f(\ell_1)+f(c)+f(r_1)+f(r_2)\geq 2k$ (resp. $15$). Since $G$ has $n/5$ branches, this shows $w(f)\geq 2kn/5$ (resp. $w(f)\geq 3n$). Now Corollary \ref{3.5} yields the result.

(3,4) Suppose $k$ is an even integer with $k\geq 8$ (resp. $k$ is an odd integer) and $G$ is a $\{P_4\}$-branch graph. Consider an arbitrary branch of $G$. For a $\gamma_k^s(G)$-function $f$, if $f(c)\leq k/2$, then $f(\ell_1)\geq k/2$ and $f(r_1)+f(r_2)\geq k$. Else $f(\ell_1)+f(c)\geq k$ and also for the right side of the branch we have $f(r_1)+f(r_2)\geq k/2$ (resp. $(k+1)/2$). Hence in each case $f(\ell_1)+f(c)+f(r_1)+f(r_2)\geq 3k/2$ (resp. $(3k+1)/2$). Since $G$ has $n/4$ branches, $w(f)\geq 3kn/8$ (resp. $(3k+1)n/8$). So the result holds by Corollary \ref{3.5}.

(5) Suppose that $G$ has $r$ branches $P_4$ and $r'$ branches $P_5$. Then $n=4r+5r'$. In view of the proof of Parts 2 and 3, for each $\gamma_8^s$-function $f$ of $G$, the weight of each branch $P_4$ is at least $12$ and the weight of each branch $P_5$ is at least $15$. Thus $w(f)\geq 12r+15r'=3(4r+5r')=3n$. Therefore Corollary \ref{3.5} completes the proof.
\end{proof}
Now we are ready to establish that the classes presented in Lemma \ref{3.6} are exactly those that achieve equality in Corollary \ref{3.5} for the strong Roman $k$-domination number. Setting $k=2$ and $k=3$ recovers \cite[Theorems 15 and 16]{double} and \cite[Theorems 2.2 and 2.3]{chambers} for Roman and double Roman domination numbers.
\begin{theorem}\label{3.7}
Assume that $G$ is a graph with $n\geq 3$ vertices. Suppose that $A_k=2k/5, \mathcal{B}_k=\{P_5\}$ for $k\in \{2,4,6\}$, $A_k=3k/8$ for other even integers $k$, $A_k=(3k+1)/8$ for odd integers $k$, $\mathcal{B}_8=\{P_4, P_5\}$ and $\mathcal{B}_k=\{P_4\}$ for other integers $k$. Then $\gamma_k^s(G)=A_kn$ if and only if
$$G= \left\lbrace \begin{array}{ll}
\mathcal{B}_k\text{-branch graph} &  \text{ when } k \text{ is odd or }k \text{ is even with } k> 8, \\
C_5 \text{ or } \mathcal{B}_k\text{-branch graph } &  \text{ when } k\in \{2,4,6,8\}.
\end{array}
\right.
$$
\end{theorem}
\begin{proof}
In view of Lemma \ref{3.6}, it suffices to prove the \textit{only if} part. To this aim we claim each tree $T$ of order $n\geq 3$  with $\gamma_k^s(T)=A_kn$  is a $\mathcal{B}_k$-branch tree. So if $G$ is a graph of order $n\geq 3$ with $\gamma_k^s(G)=A_kn$, then by Theorem \ref{3.3} and the fact that adding edges can't increase the value of strong Roman $k$-domination number,  the equality also holds for each spanning tree of $G$.  Thus our claim shows each spanning tree $T_G$ of $G$ is a $\mathcal{B}_k$-branch tree. Now if one add any edge to a spanning tree $T_G$ of $G$, one obtains either $C_5$ or a $\mathcal{B}_k$-branch graph, because adding any edge apart from joining roots of branches of $T_G$ induces another spanning tree that is not a $\mathcal{B}_k$-branch tree, a contradiction. To prove our claim by notations as in the proofs of Theorem \ref{3.3} and Lemma \ref{3.6} the equality holds in the following statements:

\textbf{Case A.}  $k\in\{2,4,6\}$. The equality holds when $T=P_5$ or in Case IV when $T'=P_5$ with $V(T')=\{w_1,v_1, u , v_0, w_0\}, E(T')=\{w_1v_1, v_1u, uv_0, v_0w_0\}$  and $T''$ has $n-5\geq 3$ vertices such that for every $k$-SRDF $f''$ for $T''$ and $f'$ for $T'$, we have $w(f'')=2k(n-5)/5$ and $w(f')=2k$ and the map $f$ defined by $f(x)=f'(x)$ for $x\in V(T')$ and $f(x)=f''(x)$ for $x\in V(T'')$ is a $\gamma_k^s(T)$-function. Thus $\gamma_k^s(T'')=2k(n-5)/5$. By induction on $n$, $T''$ is a $\{P_5\}$-branch tree. Assume $f''$ is a $\gamma_k^s(T'')$-function and $f'$ is the $\gamma_k^s(T')$-function with $f'(u)=k,f'(w_0)=f'(w_1)=k/2,f'(v_0)=f'(v_1)=0$.  Then the map $f$ defined by $f'$ on $V(T')$ and $f''$ on $V(T'')$ is a $\gamma_k^s(T)$-function and Lemma \ref{3.6} shows that the weight of $f''$ on any branch of $T''$ is at least $2k$. Now suppose that the branch $T'$ of $T$ is joined to a branch of $T''$ as in Figure \ref{fig5}. Assigning weights as shown in Figure \ref{fig5} leads to a $k$-SRDF with weight less than $w(f)$, a contradiction. So, $u$ must be adjacent to a center of a branch in $T''$. This shows $T$ is also a $\{P_5\}$-branch tree.

\begin{figure}[htbp]
\centering
\begin{tikzpicture}
[scale=0.6]
\filldraw[fill=red] (0,0) circle (2pt) node[below]{$c$};
\filldraw[fill=red] (1,1) circle (2pt) node[below]{$r_1$};
\filldraw[fill=red] (2,2) circle (2pt) node[below]{$r_2$};
\filldraw[fill=red] (-1,1) circle (2pt) node[below]{$\ell_1$};
\filldraw[fill=red] (-2,2) circle (2pt) node[below]{$\ell_2$};
\filldraw[fill=red] (0,2) circle (2pt) node[below]{$u$};
\filldraw[fill=red] (1,3) circle (2pt) node[below]{$v_0$};
\filldraw[fill=red] (2,4) circle (2pt) node[below]{$w_0$};
\filldraw[fill=red] (-1,3) circle (2pt) node[below]{$v_1$};
\filldraw[fill=red] (-2,4) circle (2pt) node[below]{$w_1$};

\filldraw[fill=red] (0,-0.5) node[blue,below]{{\tiny $0$}};
\filldraw[fill=red] (1,0.5) node[blue,below]{$\frac{k}{2}$};
\filldraw[fill=red] (2,1.5) node[blue,below]{{\tiny $0$}};
\filldraw[fill=red] (-1,0.3) node[blue,below]{{\tiny $k$}};
\filldraw[fill=red] (-2,1.3)  node[blue,below]{{\tiny $0$}};
\filldraw[fill=red] (0,2) node[blue,above]{{\tiny $k$}};
\filldraw[fill=red] (1,3) node[blue,above]{{\tiny $0$}};
\filldraw[fill=red] (2,4)  node[blue,above]{$\frac{k}{2}$};
\filldraw[fill=red] (-1,3)  node[blue,above]{{\tiny $0$}};
\filldraw[fill=red] (-2,4)  node[blue,above]{$\frac{k}{2}$};

\draw (0,0) -- (1,1);
\draw (1,1) -- (2,2);
\draw (0,0) -- (-1,1);
\draw (-1,1) -- (-2,2);
\draw (0,2) -- (1,3);
\draw (1,3) -- (2,4);
\draw (0,2) -- (-1,3);
\draw (-1,3) -- (-2,4);

\draw (0,2) -- (2,2);
\end{tikzpicture}
\hspace{1cm}
\begin{tikzpicture}
[scale=0.6]
\filldraw[fill=red] (0,0) circle (2pt) node[below]{$c$};
\filldraw[fill=red] (1,1) circle (2pt) node[below]{$r_1$};
\filldraw[fill=red] (2,2) circle (2pt) node[below]{$r_2$};
\filldraw[fill=red] (-1,1) circle (2pt) node[below]{$\ell_1$};
\filldraw[fill=red] (-2,2) circle (2pt) node[below]{$\ell_2$};
\filldraw[fill=red] (0,2) circle (2pt) node[below]{$u$};
\filldraw[fill=red] (1,3) circle (2pt) node[below]{$v_0$};
\filldraw[fill=red] (2,4) circle (2pt) node[below]{$w_0$};
\filldraw[fill=red] (-1,3) circle (2pt) node[below]{$v_1$};
\filldraw[fill=red] (-2,4) circle (2pt) node[below]{$w_1$};

\filldraw[fill=red] (0,-0.5) node[blue,below]{{\tiny $0$}};
\filldraw[fill=red] (1,0.5) node[blue,below]{{\tiny $0$}};
\filldraw[fill=red] (2,1.5) node[blue,below]{$\frac{k}{2}$};
\filldraw[fill=red] (-1,0.3) node[blue,below]{{\tiny $k$}};
\filldraw[fill=red] (-2,1.3)  node[blue,below]{{\tiny $0$}};
\filldraw[fill=red] (0,2) node[blue,above]{{\tiny $k$}};
\filldraw[fill=red] (1,3) node[blue,above]{{\tiny $0$}};
\filldraw[fill=red] (2,4)  node[blue,above]{$\frac{k}{2}$};
\filldraw[fill=red] (-1,3)  node[blue,above]{{\tiny $0$}};
\filldraw[fill=red] (-2,4)  node[blue,above]{$\frac{k}{2}$};

\draw (0,0) -- (1,1);
\draw (1,1) -- (2,2);
\draw (0,0) -- (-1,1);
\draw (-1,1) -- (-2,2);
\draw (0,2) -- (1,3);
\draw (1,3) -- (2,4);
\draw (0,2) -- (-1,3);
\draw (-1,3) -- (-2,4);

\draw (0,2) -- (1,1);
\end{tikzpicture}

\caption{Proof of Theorem \ref{3.7} (Cases A and C)}
\label{fig5}
\end{figure}
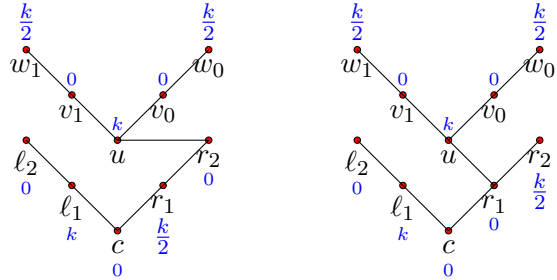

\begin{figure}[htbp]
\centering
\begin{tikzpicture}
[scale=0.6]
\filldraw[fill=red] (0,0) circle (2pt) node[below]{$c$};
\filldraw[fill=red] (1,1) circle (2pt) node[below]{$r_1$};
\filldraw[fill=red] (2,2) circle (2pt) node[below]{$r_2$};
\filldraw[fill=red] (-1,1) circle (2pt) node[below]{$\ell_1$};
\filldraw[fill=red] (0,2) circle (2pt) node[below]{$u$};
\filldraw[fill=red] (1,3) circle (2pt) node[below]{$v_0$};
\filldraw[fill=red] (2,4) circle (2pt) node[below]{$w_0$};
\filldraw[fill=red] (-1,3) circle (2pt) node[below]{$u_1$};

\filldraw[fill=red] (0,-0.5) node[blue,below]{{\tiny $\frac{k}{2}+1$}};
\filldraw[fill=red] (1,0.5) node[blue,below]{{\tiny $0$}};
\filldraw[fill=red] (2,1.5) node[blue,below]{{\tiny $\frac{k}{2}+1$}};
\filldraw[fill=red] (-1,0.3) node[blue,below]{{\tiny $0$}};
\filldraw[fill=red] (0,2) node[blue,above]{{\tiny $k$}};
\filldraw[fill=red] (1,3) node[blue,above]{{\tiny $0$}};
\filldraw[fill=red] (2,4)  node[blue,above]{{\tiny $\frac{k}{2}$}};
\filldraw[fill=red] (-1,3)  node[blue,above]{{\tiny $0$}};

\draw (0,0) -- (1,1);
\draw (1,1) -- (2,2);
\draw (0,0) -- (-1,1);
\draw (0,2) -- (1,3);
\draw (1,3) -- (2,4);
\draw (0,2) -- (-1,3);

\draw (0,2) -- (-1,1);
\end{tikzpicture}
\hspace{2cm}
\begin{tikzpicture}
[scale=0.6]
\filldraw[fill=red] (0,0) circle (2pt) node[below]{$c$};
\filldraw[fill=red] (1,1) circle (2pt) node[below]{$r_1$};
\filldraw[fill=red] (2,2) circle (2pt) node[below]{$r_2$};
\filldraw[fill=red] (-1,1) circle (2pt) node[below]{$\ell_1$};
\filldraw[fill=red] (0,2) circle (2pt) node[below]{$u$};
\filldraw[fill=red] (1,3) circle (2pt) node[below]{$v_0$};
\filldraw[fill=red] (2,4) circle (2pt) node[below]{$w_0$};
\filldraw[fill=red] (-1,3) circle (2pt) node[below]{$u_1$};

\filldraw[fill=red] (0,-0.5) node[blue,below]{{\tiny $\frac{k+1}{2}$}};
\filldraw[fill=red] (1,0.5) node[blue,below]{{\tiny $0$}};
\filldraw[fill=red] (2,1.5) node[blue,below]{{\tiny $\frac{k+1}{2}$}};
\filldraw[fill=red] (-1,0.3) node[blue,below]{{\tiny $0$}};
\filldraw[fill=red] (0,2) node[blue,above]{{\tiny $k$}};
\filldraw[fill=red] (1,3) node[blue,above]{{\tiny $0$}};
\filldraw[fill=red] (2,4)  node[blue,above]{{\tiny $\frac{k+1}{2}$}};
\filldraw[fill=red] (-1,3)  node[blue,above]{{\tiny $0$}};

\draw (0,0) -- (1,1);
\draw (1,1) -- (2,2);
\draw (0,0) -- (-1,1);
\draw (0,2) -- (1,3);
\draw (1,3) -- (2,4);
\draw (0,2) -- (-1,3);

\draw (0,2) -- (-1,1);
\end{tikzpicture}

{\hspace{-15mm}($k>8$ and $k$ is even) \hspace{1.5cm} ($k$ is odd)}

\caption{Proof of Theorem \ref{3.7} (Cases B and C)}
\label{fig6}
\end{figure}
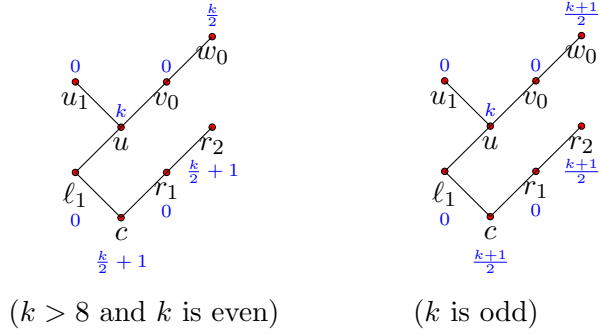

\begin{figure}[htbp]
\centering
\begin{tikzpicture}
[scale=0.6]
\filldraw[fill=red] (0,0) circle (2pt) node[below]{$c$};
\filldraw[fill=red] (1,1) circle (2pt) node[below]{$r_1$};
\filldraw[fill=red] (2,2) circle (2pt) node[below]{$r_2$};
\filldraw[fill=red] (-1,1) circle (2pt) node[below]{$\ell_1$};
\filldraw[fill=red] (-2,2) circle (2pt) node[below]{$\ell_2$};
\filldraw[fill=red] (0,2) circle (2pt) node[below]{$u$};
\filldraw[fill=red] (1,3) circle (2pt) node[below]{$v_0$};
\filldraw[fill=red] (2,4) circle (2pt) node[below]{$w_0$};
\filldraw[fill=red] (-1,3) circle (2pt) node[below]{$u_1$};

\filldraw[fill=red] (0,-0.3) node[blue,below]{{\tiny $\frac{k}{2}+1$}};
\filldraw[fill=red] (1,0.5) node[blue,below]{{\tiny $0$}};
\filldraw[fill=red] (2,1.5) node[blue,below]{{\tiny $\frac{k}{2}+1$}};
\filldraw[fill=red] (-1,0.4) node[blue,below]{{\tiny $0$}};
\filldraw[fill=red] (-2,1.4) node[blue,below]{{\tiny $\frac{k}{2}$}};
\filldraw[fill=red] (0,2) node[blue,above]{{\tiny $k$}};
\filldraw[fill=red] (1,3) node[blue,above]{{\tiny $0$}};
\filldraw[fill=red] (2,4)  node[blue,above]{{\tiny $\frac{k}{2}$}};
\filldraw[fill=red] (-1,3)  node[blue,above]{{\tiny $0$}};

\draw (0,0) -- (1,1);
\draw (1,1) -- (2,2);
\draw (0,0) -- (-1,1);
\draw (-1,1) -- (-2, 2);
\draw (0,2) -- (1,3);
\draw (1,3) -- (2,4);
\draw (0,2) -- (-1,3);

\draw (0,2) -- (-1,1);
\end{tikzpicture}
\hspace{2cm}
\begin{tikzpicture}
[scale=0.6]
\filldraw[fill=red] (0,0) circle (2pt) node[below]{$c$};
\filldraw[fill=red] (1,1) circle (2pt) node[below]{$r_1$};
\filldraw[fill=red] (2,2) circle (2pt) node[below]{$r_2$};
\filldraw[fill=red] (-1,1) circle (2pt) node[below]{$\ell_1$};
\filldraw[fill=red] (-2,2) circle (2pt) node[below]{$\ell_2$};
\filldraw[fill=red] (0,2) circle (2pt) node[below]{$u$};
\filldraw[fill=red] (1,3) circle (2pt) node[below]{$v_0$};
\filldraw[fill=red] (2,4) circle (2pt) node[below]{$w_0$};
\filldraw[fill=red] (-1,3) circle (2pt) node[below]{$u_1$};

\filldraw[fill=red] (0,-0.3) node[blue,below]{{\tiny $\frac{k}{2}+1$}};
\filldraw[fill=red] (1,0.5) node[blue,below]{{\tiny $0$}};
\filldraw[fill=red] (2,1.5) node[blue,below]{{\tiny $\frac{k}{2}+1$}};
\filldraw[fill=red] (-1,0.4) node[blue,below]{{\tiny $\frac{k}{2}$}};
\filldraw[fill=red] (-2,1.4) node[blue,below]{{\tiny $0$}};
\filldraw[fill=red] (0,2) node[blue,above]{{\tiny $k$}};
\filldraw[fill=red] (1,3) node[blue,above]{{\tiny $0$}};
\filldraw[fill=red] (2,4)  node[blue,above]{{\tiny $\frac{k}{2}$}};
\filldraw[fill=red] (-1,3)  node[blue,above]{{\tiny $0$}};

\draw (0,0) -- (1,1);
\draw (1,1) -- (2,2);
\draw (0,0) -- (-1,1);
\draw (-1,1) -- (-2,2);
\draw (0,2) -- (1,3);
\draw (1,3) -- (2,4);
\draw (0,2) -- (-1,3);

\draw (0,2) -- (-2,2);
\end{tikzpicture}
\hspace{2cm}
\begin{tikzpicture}
[scale=0.6]
\filldraw[fill=red] (0,0) circle (2pt) node[below]{$c$};
\filldraw[fill=red] (1,1) circle (2pt) node[below]{$r_1$};
\filldraw[fill=red] (2,2) circle (2pt) node[below]{$r_2$};
\filldraw[fill=red] (-1,1) circle (2pt) node[below]{$\ell_1$};
\filldraw[fill=red] (0,2) circle (2pt) node[below]{$u$};
\filldraw[fill=red] (1,3) circle (2pt) node[below]{$v_0$};
\filldraw[fill=red] (2,4) circle (2pt) node[below]{$w_0$};
\filldraw[fill=red] (-1,3) circle (2pt) node[below]{$v_1$};
\filldraw[fill=red] (-2,4) circle (2pt) node[below]{$w_1$};

\filldraw[fill=red] (0,-0.3) node[blue,below]{{\tiny $\frac{k}{2}+1$}};
\filldraw[fill=red] (1,0.5) node[blue,below]{{\tiny $0$}};
\filldraw[fill=red] (2,1.5) node[blue,below]{{\tiny $\frac{k}{2}+1$}};
\filldraw[fill=red] (-1,0.4) node[blue,below]{{\tiny $0$}};
\filldraw[fill=red] (0,2) node[blue,above]{{\tiny $k$}};
\filldraw[fill=red] (1,3) node[blue,above]{{\tiny $0$}};
\filldraw[fill=red] (2,4)  node[blue,above]{{\tiny $\frac{k}{2}$}};
\filldraw[fill=red] (-1,3)  node[blue,above]{{\tiny $0$}};
\filldraw[fill=red] (-2,4)  node[blue,above]{{\tiny $\frac{k}{2}$}};

\draw (0,0) -- (1,1);
\draw (1,1) -- (2,2);
\draw (0,0) -- (-1,1);
\draw (0,2) -- (1,3);
\draw (1,3) -- (2,4);
\draw (0,2) -- (-1,3);
\draw (-1,3) -- (-2,4);

\draw (0,2) -- (-1,1);
\end{tikzpicture}

\caption{Proof of Theorem \ref{3.7} (Case C)}
\label{fig7}
\end{figure}
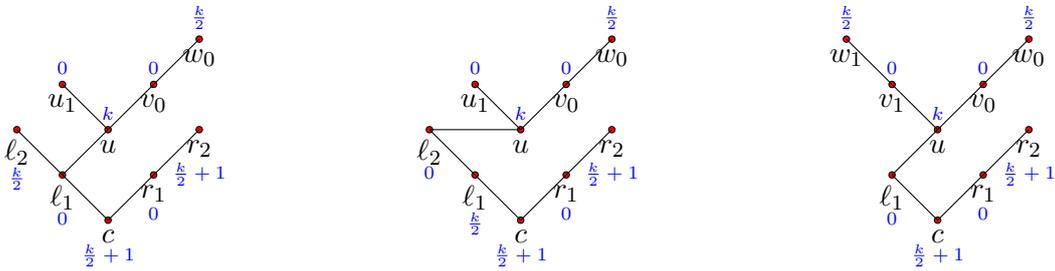

\textbf{Case B. } $k$ is even and $k>8$ (resp. $k$ is odd).  The equality holds when $T=P_4$ or in Case IV when  $T'=P_4$ with $V(T')=\{u_1, u , v_0, w_0\}, E(T')=\{u_1u, uv_0, v_0w_0\}$  and $T''$ has $n-4\geq 3$ vertices such that the equalities in proof of Theorem \ref{3.3} hold. The proof is similar to that of Case A (Figure \ref{fig6} shows that if $u$ is not adjacent to the center of a branch of $T''$, then the weight of $f''$ on that branch is less than $3k/2$ (resp. $(3k+1)/2$) which is a contradiction).

\textbf{Case C.}  $k=8$. The equality holds when $G=P_4$ or $G=P_5$ or in Case IV when for $r\in \{4,5\}$,  $T'=P_r$ and $T''$ has $n-r\geq 3$ vertices such that the equalities in proof of Theorem \ref{3.3}  hold. The proof is similar to that of Case A (Figures \ref{fig5}, \ref{fig6} and \ref{fig7} show that if $u$ is not adjacent to the center of a branch in $T''$, then the weight of $f''$ on that branch, in the form of $P_4$, is less than $3k/2$  and in the form of $P_5$, is less than $2k$, a contradiction).
\end{proof}

We conclude by presenting a class of graphs whose Roman $k$-domination number precisely matches the bound in Corollary \ref{3.5} generalizing the extremal parts of Theorems 1 and 12 in \cite{Klos}.
\begin{theorem}\label{3.9}
Let $n\geq 3$ and $G$ be a graph with $n$ vertices. Set $A_k=3k/8$ when $k$ is even and $A_k=(3k+1)/8$ when $k$ is odd. Then $\gamma_k(G)=A_kn$ if and only if $G$ is a $\{P_4\}$-branch graph.
\end{theorem}
\begin{proof}
$(\Longleftarrow )$ Suppose that $f$ is a $\gamma_k(G)$-function. In view of Corollary \ref{3.5} it suffices to prove $w(f)\geq A_kn$. To this end, we show that the weight of $f$ on each branch of $G$ is at least $4A_k$.  Consider $P_4$ as 
$$V(P_4)=\{\ell_1, c, r_1, r_2\}, E(P_4)=\{\ell_1c, cr_1, r_1r_2\}.$$

If $f(c)\geq k/2$, then $f(\ell_1)+f(c)\geq k$ and also either $f(r_2)\geq k/2$ or $f(r_1)+f(r_2)\geq k$. Hence the weight of $f$ on the branch is at least $4A_k$. Else $f(c)<k/2$. Thus $f(\ell_1)\geq k/2$. In addition, it can be checked $f(r_1)+f(r_2)+f(c)\geq k$ which implies that the weight of $f$ on the branch is at least $4A_k$ in this case too.

$(\Longrightarrow )$  Similar argument to proof of Theorem \ref{3.7} yields the result. 
\end{proof}

\begin{remark}
This work unifies and extends sharp upper bounds for Roman and strong Roman $k$-domination numbers, generalizing prior results for $k=2$ and $k=3$ to arbitrary $k \geq 2$, with new insights for $k \geq 4$ and precise characterizations of extremal graphs. The unification of diverse domination variants (e.g., perfect, weak, and etc.) remains an open challenge, offering a promising direction for further research into generalized domination numbers across arbitrary integers $k$.
\end{remark}

\textbf{Acknowledgment.} The author sincerely thanks the referees for their valuable and thoughtful comments, which substantially improved the exposition and presentation of this note.

\end{document}